\documentclass[11pt,reqno]{amsart}
\usepackage{a4wide}
\usepackage{amssymb}
\usepackage{amsthm}
\usepackage{amsmath,todonotes}
\usepackage{amscd}
\usepackage{cancel}
\usepackage{verbatim}
\usepackage{color}
\usepackage[all]{xy}
\usepackage[mathscr]{eucal}
\usepackage{mathrsfs}
\usepackage{fullpage}
\usepackage{enumitem}
\usepackage{hyperref}
\usepackage{bbm}
\usepackage{qtree}
\usepackage{bm}
\usepackage{bigints}
\usepackage[headheight=110pt,top=.95in, bottom=.95in, left=0.85in, right=0.85in]{geometry}
\allowdisplaybreaks

\numberwithin{equation}{section}

\newtheorem{theorem}{Theorem}

\newtheorem{lemma}[theorem]{Lemma}

\theoremstyle{remark}

\theoremstyle{definition}

\numberwithin{theorem}{section} 
\numberwithin{equation}{section}
\numberwithin{table}{section}

\newcommand{\re}{\textnormal{Re}}
\newcommand{\im}{\textnormal{Im}}

\renewcommand{\(}{\left(}
\renewcommand{\)}{\right)}

\begin{document}
\title[Equivalent criterion for the grand Riemann hypothesis]{Equivalent criterion for the grand Riemann hypothesis associated to Maass cusp forms}
\author{Soumyarup Banerjee}
\address{Discipline of Mathematics, Indian Institute of Technology Gandhinagar, Palaj, Gandhinagar-382355, Gujarat, India.}
\email{soumyarup.b@iitgn.ac.in}
\author{Rahul Kumar}
\address{Center for Geometry and Physics, Institute for Basic Science (IBS), Pohang 37673, Republic of Korea.}
\email{rahul@ibs.re.kr}

\medskip
\begin{abstract}
In this article, we obtain transformation formulas analogous to the identity of Ramanujan, Hardy and Littlewood in the setting of primitive Maass cusp form over the congruence subgroup $\Gamma_0(N)$ and also provide an equivalent criterion of the grand Riemann hypothesis for the $L$-function associated to the primitive Maass cusp form over $\Gamma_0(N)$.
\end{abstract}

\thanks{}
\maketitle

\section{Introduction}\label{Introduction}
Let $f$ be a primitive Maass cusp form over the congruence subgroup $\Gamma_0(N) \subset \Gamma = SL_2(\mathbb{Z})$ with $\Delta f = (\frac{1}{4}+\nu^2)f$, where $\Delta$ is the non-Euclidean Laplacian operator
$$\Delta = -y^2\(\frac{\partial^2}{\partial^2 x}+\frac{\partial^2}{\partial^2 y}\right).$$
Since the Laplacian is self-adjoint, any eigenvalue associated to $\Delta$ must be a positive real number, and therefore $\nu$ is either real or of the form $it$ with $t$ real and $-1/2<t<1/2$. 


We denote the Fourier coefficients of $f$ at cusp infinity by $\lambda_f(n)$ :
\begin{equation*}
f(z) = y^{1/2}\sum_{n\neq 0} \lambda_f(n) K_{i\nu}(2\pi|n|y)e(nx)
\end{equation*}
normalized by setting $\lambda_f(1)=1$. Here $\lambda_f(n)$ is then the $n$th Hecke eigenvalue of $f$, $K_{i\nu}$ is the modified Bessel function of second kind, $z=x+iy$ and $e(x) = e^{2\pi i x}$. The Ramanujan conjecture asserts that $|\lambda_f(n)| = \mathcal{O}(n^\epsilon)$ for any positive $\epsilon$. In this direction, to the best our knowledge, the best known bound for the Fourier coefficients was given by Kim and Sarnak \cite{Kim}, which is
\begin{equation}\label{Sarnak bound}
|\lambda_f(n)| = \mathcal{O}\(n^{7/64+\epsilon}\)
\end{equation}
for any positive $\epsilon$.

Let $\rho : \mathbb{H} \to \mathbb{H}$ be the antiholomorphic involution given by $\rho(x+iy) = -x+iy$. We call a Maass cusp form $f$ even if $f \circ \rho = f$ and odd if $f \circ \rho = -f$. We may attach an automorphic $L$-function associated to the primitive Maass cusp form $f$ over $\Gamma_0(N)$  as
\begin{equation}\label{L-function}
L(s, f) := \sum_{n=1}^\infty \frac{\lambda_f(n)}{n^s}, \hspace{1cm} \re(s)>1.
\end{equation}
It has an Euler product of the form (cf. \cite[p. 208]{Liu})
\begin{equation}\label{Euler product}
L(s, f) = \prod_p \left(1-\frac{\lambda_f(p)}{p^s} + \frac{\chi_0(p)}{p^{2s}} \right)^{-1}
\end{equation}
for $\re(s)>1$, where $\chi_0$ denotes the principle character modulo $N$. The $L$-function $L(s, f)$ can be analytically continued to the entire complex plane. The completed $L$-function associated to $f$ can be defined as
\begin{equation*}
\Lambda(s, f) = L_{\infty}(s, f) L(s, f),
\end{equation*}
where 
\begin{equation}\label{Gamma factors in FE}
L_{\infty}(s, f) := \pi^{-s} \Gamma\(\frac{s+\epsilon+i\nu}{2} \right) \Gamma\(\frac{s+\epsilon-i\nu}{2} \right).
\end{equation}
Here $\epsilon$ takes the value $0$ for $f$ even and $1$ for $f$ odd. The completed $L$-function $\Lambda(s, f)$ satisfies the functional equation \cite[p. 208]{Liu}
\begin{equation}\label{Functional equation}
\Lambda(s, f) = \epsilon_f N^{1/2-s} \Lambda(1-s, f),
\end{equation} 
where $\epsilon_f$ is a complex number of modulus $1$. The non-trivial zeros of $L(s, f)$ lie in the critical strip $0<\re(s)<1$ and, by the grand Riemann hypothesis, are conjectured to be on the critical line $\re(s) = \frac{1}{2}$. 
The goal of this article is to obtain a transformation formula involving the non-trivial zeros of $L(s, f)$ and to find an equivalent criterion of the grand Riemann hypothesis for $L(s, f)$.

In the classical case, the problem was initiated by Ramanujan. He showed an elegant transformation formula to Hardy involving an infinite series of the M{\"o}bius function  \cite[p. 312]{Ramanujan}, during his stay at Trinity. Later, the identity was corrected by Hardy and Littlewood in \cite[p. 156, Section 2.5]{Hardy}, which precisely states that for $\alpha$ and $\beta$ being two positive real numbers with $\alpha \beta = \pi$, the following identity holds :
\begin{equation}\label{Ramanujan-Hardy-Littlewood identity}
\sqrt{\alpha}\sum_{n=1}^\infty \frac{\mu(n)}{n}e^{-\alpha^2/n^2} - \sqrt{\beta}\sum_{n=1}^\infty \frac{\mu(n)}{n}e^{-\beta^2/n^2} = - \frac{1}{2\sqrt{\beta}}\sum_\rho\frac{\Gamma\( \frac{1-\rho}{2}\right)}{\zeta'(\rho)}\beta^{\rho},
\end{equation}
provided the infinite series on the right hand side of the above equation converges where the sum is running over the non-trivial zeros $\rho$ of the Riemann zeta function with an assumption that the non-trivial zeros are simple. For more details, one can look into \cite[p. 467--468]{Berndt}.

The series on the right hand side of \eqref{Ramanujan-Hardy-Littlewood identity} converges if we bracket the terms of the series in such a way that the terms for which
\begin{equation*}
|\im(\rho) - \im(\rho')| < \exp \(-c \frac{\im(\rho) }{\log(\im(\rho))} \right) + \exp \(-c \frac{\im(\rho') }{\log(\im(\rho'))} \right)
\end{equation*}
are included in the same bracket (cf. \cite[p. 220]{Titchmarsh}) but the convergence of the series  is still unknown unconditionally. 

The identity \eqref{Ramanujan-Hardy-Littlewood identity} leads Hardy and Littlewood \cite{Hardy} to obtain an equivalent criterion of the Riemann hypothesis for $\zeta(s)$, which precisely states that for the function $P(y) := \sum_{n=1}^\infty \frac{(-y)^n}{n!\zeta(2n+1)}$, the bound $P(y) = \mathcal{O}(y^{-1/4+\delta})$ as $y\to \infty$ for any positive $\delta$ is equivalent to the Riemann hypothesis. This equivalent criterion for the Riemann hypothesis is sometimes known as Riesz-type criterion since Riesz \cite{Riesz} was the first mathematician to find the similar result around the same time.

Later, several analogues and generalizations have been investigated in different directions. For instance, the similar problem involving an extra complex variable was considered in \cite{Dixit2}. The problem in the setting of Dirichlet $L$-function and Dedekind zeta function has been studied in \cite{Dix} and \cite{Dixit3} respectively. In \cite{Dixit}, Dixit et al. have obtained an identity analogous to \eqref{Ramanujan-Hardy-Littlewood identity} for holomorphic Hecke eigenforms and found the similar criterion of the Riemann hypothesis for the associated $L$-function. Recently, in \cite{Agarwal}, one variable generalization of \eqref{Ramanujan-Hardy-Littlewood identity} have been studied in different direction.

We next provide the analogues of \eqref{Ramanujan-Hardy-Littlewood identity} for the non-holomorphic primitive Maass cusp form over the congruence subgroup $\Gamma_0(N)\subset \Gamma$. The reciprocal of $L(s, f)$ for $\re(s)>1$ can be written as 
\begin{equation*}
\frac{1}{L(s, f)} = \sum_{n=1}^\infty \frac{\widetilde{\lambda}_f(n)}{n^s}.
\end{equation*} 
where
\begin{equation}\label{Lambda tilda}
\widetilde{\lambda}_f(n) = \begin{cases}
\mu(d)\chi_0(D)\lambda_f(d) & \text{ if } n=dD^2, \, (d, D) = 1 \text{ and } d, D \text{ squarefree} \\
0 & \text{ otherwise. }
\end{cases}
\end{equation}
The above expression of $\widetilde{\lambda}_f(n)$ can be derived mainly from the Euler product of $L(s, f)^{-1}$, which follows from \eqref{Euler product} that 
\begin{equation}\label{Euler product of reciprocal}
\sum_{n=1}^\infty \frac{\widetilde{\lambda}_f(n)}{n^s} = \prod_p \(1- \frac{\lambda_f(p)}{p^s}+\frac{\chi_0(p)}{p^{2s}} \right).
\end{equation}
The multiplicativity of $\widetilde{\lambda}_f(n)$ and \eqref{Euler product of reciprocal} together implies that for $p^m | n$ with $m\geq 3$, we have $\widetilde{\lambda}_f(n) = 0$. Therefore, from the fact that both $\lambda_f(n)$ and $\chi_0(n)$ are multiplicative, we obtain \eqref{Lambda tilda}.

We next define two functions
\begin{align}\label{Odd function}
\mathcal{P}_{f_o} (y) := \sum_{n=1}^{\infty} \frac{\widetilde{\lambda}_f(n)}{n^2} K_{i\nu}\left(\frac{2\pi y}{n}\right)
\end{align}
and
\begin{align}\label{even function}
\mathcal{P}_{f_e}(y) := \sum_{n=1}^{\infty} \frac{\widetilde{\lambda}_f(n)}{n} \left[2K_{i\nu}\left(\frac{2\pi y}{n}\right)-\frac{\Gamma(i\nu)}{(\pi y/n)^{i\nu}} - \frac{\Gamma(-i\nu)}{(\pi y/n)^{-i\nu}}\right]. 
\end{align}
In the following theorem, we obtain transformation formulas involving the non-trivial zeros of $L(s, f)$.
\begin{theorem}\label{Hardy type identity}
Let $f$ be a normalized primitive Maass cusp form over the congruence subgroup $\Gamma_0(N)$ with eigenvalue $1/4+\nu^2$. Let $L(s, f)$ be its associated $L$-function as defined in \eqref{L-function} with an assumption that all its non-trivial zeros are simple. Let $\epsilon_f$ be the complex number of modulus $1$ appeared in \eqref{Functional equation}. If $\alpha$ and $\beta$ are any positive numbers with $\alpha\beta = 1/N$. Then
\begin{itemize}
\item[{\rm (a)}] for $f$ odd, we have
\begin{equation}\label{Odd identity}
\alpha^{3/2} \mathcal{P}_{f_o} (\alpha) - \epsilon_f \beta^{3/2} \mathcal{P}_{f_o} (\beta) = -\frac{\epsilon_f}{4\pi}\sum_{\rho} \frac{L_\infty(1-\rho, f)}{L'(\rho, f)} \beta^{\rho-\frac{1}{2}},
\end{equation}
\item[{\rm (b)}] for $f$ even, we have
\begin{align}\label{Even identity}
\sqrt{\alpha} \mathcal{P}_{f_e}(\alpha) - \epsilon_f \sqrt{\beta} \mathcal{P}_{f_e}(\beta) = \frac{\epsilon_f \pi^{i\nu}\Gamma(-i\nu)}{L(-i\nu,f)}\beta^{\frac{1}{2}+i\nu}+\frac{\epsilon_f \pi^{-i\nu}\Gamma(i\nu)}{L(i\nu,f)}\beta^{\frac{1}{2}-i\nu} 
-\epsilon_f\sum_\rho \frac{L_\infty(1-\rho, f)}{L'(\rho, f)} \beta^{\rho-\frac{1}{2}},
\end{align}
\end{itemize}
provided the series on the right-hand of \eqref{Odd identity} and \eqref{Even identity} converge where $\rho$ runs through the non-trivial zeros of $L(s, f)$.
\end{theorem}
The next theorem provides an equivalent criterion of the grand Riemann hypothesis for $L(s, f)$ when $f$ is odd, motivated from the above theorem. 
\begin{theorem}\label{Equivalent criterion}
If $f$ is an odd normalized primitive Maass cusp form over the congruence subgroup $\Gamma_0(N)$ with eigenvalue $1/4+\nu^2$. Then the bound $\mathcal{P}_{f_o} (y) = \mathcal{O}\(y^{-3/2+\delta}\right)$  
as $y \to \infty$ for every positive $\delta$ is equivalent to the grand Riemann hypothesis for $L(s, f)$.
\end{theorem}
Now we consider the remaining case when $f$ is even. In obtaining the criterion of the grand Riemann hypothesis for $L(s, f)$ for $f$ even, we need to consider the derivative of the function $P_{f_e}(y)$, unlike to the case when $f$ is odd. The reason behind the above fact is technical, which is mainly due to the poles coming from the gamma factors involved in the functional equation \eqref{Functional equation} of $L(s, f)$. Let $\mathcal{Q}_{f_e} (y)$ be the derivative of the function $P_{f_e}(y)$. Thus it follows from the derivative of the K-Bessel function \cite[p.~36, Formula 1.14.1.1]{Brychkov} that
\begin{align}\label{Q function}
\mathcal{Q}_{f_e} (y) = \pi \sum_{n=1}^\infty \frac{\widetilde{\lambda}_f(n)}{n^2} \left[ 2K_{1+i\nu}\left(\frac{2\pi y}{n}\right)+2K_{1-i\nu}\left(\frac{2\pi y}{n}\right) - \frac{\Gamma(1+i\nu)}{(\pi y/n)^{1+i\nu}} - \frac{\Gamma(1-i\nu)}{(\pi y/n)^{1-i\nu}}\right].
\end{align}
\begin{theorem}\label{Equivalent criterion even}
Let $f$ be an even normalized primitive Maass cusp form over the congruence subgroup $\Gamma_0(N)$ with eigenvalue $1/4+\nu^2$. Then 
\begin{itemize}
\item[{\rm (a)}] The bound $\mathcal{Q}_{f_e} (y) = \mathcal{O}\(y^{-3/2+\delta}\right)$ for any $\delta>0$ implies the grand Riemann hypothesis for $L(s, f)$.
\item[{\rm (b)}] If the grand Riemann hypothesis for $L(s, f)$ is true, then as $y\to \infty$,
\begin{equation}\label{Bound on Q}
\mathcal{Q}_{f_e} (y) = -\pi \sum_{n=1}^{[y^{1-\epsilon}]-1}\frac{\widetilde{\lambda}_f(n)}{n^2} \left\lbrace \frac{\Gamma(1+i\nu)}{(\pi y/n)^{1+i\nu}} + \frac{\Gamma(1-i\nu)}{(\pi y/n)^{1-i\nu}} \right\rbrace + \mathcal{O}\left(y^{-3/2+\delta}\right)
\end{equation}
and
\begin{equation*}
\mathcal{P}_{f_e} (y) = - \sum_{n=1}^{[y^{1-\epsilon}]-1}\frac{\widetilde{\lambda}_f(n)}{n} \left\lbrace \frac{\Gamma(i\nu)}{(\pi y/n)^{i\nu}} + \frac{\Gamma(-i\nu)}{(\pi y/n)^{-i\nu}} \right\rbrace + \mathcal{O}\left(y^{-1/2+\delta}\right).
\end{equation*}
\end{itemize}
\end{theorem}
The article is organized as follows. The next section is devoted to obtaining the transformation formulas of $L(s, f)$. We prove the criterion for the grand Riemann hypothesis for $L(s, f)$ for both odd and even case in \S \ref{Odd criterion} and \S \ref{Even criterion} respectively.

\section{Transformation formula involving the non-trivial zeros of $L(s, f)$}
In this section, we obtain an analogous result of \eqref{Ramanujan-Hardy-Littlewood identity} in the setting of primitive Maass cusp form. In the following lemma, we provide a lower bound for $L(s, f)$ which will be used to prove Theorem \ref{Hardy type identity}.
\begin{lemma}\label{Lower bound of L(s, f)}
For a sequence of positive numbers $T$ as $T\to \infty$ through values such that $|T-\gamma| > \exp (-A_1\gamma/\log \gamma)$ for every ordinate $\gamma$ of a zero of $L(s, f)$, where $A_1$ is sufficiently small positive constant, the bound
\begin{equation*}
|L(\sigma + iT, f)| \geq e^{-A_2 T}
\end{equation*}
holds for $\sigma \in [-1/2, 3/2]$, where $0<A_2<\frac{\pi}{4}$.
\end{lemma}
\begin{proof}
The key ingredients to prove the above lemma are followed from \cite[ p.~102, Fomula 5.24]{Iwaniec} and \cite[Lemma 3.6]{Sankaranarayanan}. One can therefore apply the similar argument as in \cite[p.~219, Section 9.8]{Titchmarsh} or \cite[p. 6]{Dixit3} to complete the proof. 
\end{proof}
We next provide the proof of Theorem \ref{Hardy type identity} only for the case when $f$ odd since the proof for $f$ even goes along the similar direction.

\subsection{Proof of Theorem \ref{Hardy type identity}}
The inverse Mellin transform of the modified Bessel function $K_{\mu}(x)$ \cite[p.~253, Formula 10.32.13]{temme} is given by
\begin{align}\label{imt of kinu}
K_{\mu}(x)=\frac{1}{2\pi i}\int_{(c)}2^{s-2}\Gamma\left(\frac{s-\mu}{2}\right)\Gamma\left(\frac{s+\mu}{2}\right)x^{-s}ds,
\end{align}
which is valid for any $c >\pm\mathrm{Re}(\mu)$. Here and throughout the article, $\int_{(c)}$ denotes the integral $\int_{c-i\infty}^{c+i\infty}$. Replacing $s$ by $s+1$ in \eqref{imt of kinu} and letting $x=2\pi \alpha/n$ and $\mu=i\nu$, 
we have
\begin{align}\label{k mellin}
\frac{4\pi\alpha}{n}K_{i\nu}\left(\frac{2\pi\alpha}{n}\right)=\frac{1}{2\pi i}\int_{(c)}L_\infty(s,f)\left(\frac{\alpha}{n}\right)^{-s}ds,
\end{align}
for $-\frac{1}{2}<c<0$, where $L_\infty(s,f)$ is defined in \eqref{Gamma factors in FE} for $f$ odd. 
We next insert \eqref{k mellin} into \eqref{Odd function} to obtain
\begin{equation}\label{Integral representation of odd function}
4\pi\alpha \mathcal{P}_{f_o} (\alpha) = \sum_{n=1}^\infty\frac{\widetilde{\lambda}_f(n)}{n} \frac{1}{2\pi i}\int_{(c)}L_\infty(s,f)\left(\frac{\alpha}{n}\right)^{-s}ds = \frac{1}{2\pi i}\int_{(c)}\frac{L_\infty(s,f)}{L(1-s,f)}\alpha^{-s}ds,
\end{equation}
where in the last step we have interchanged the order of summation and integration. The functional equation of $L(s, f)$, as in \eqref{Functional equation} thus yields
\begin{equation}\label{Apply functional equation}
4\pi\alpha \mathcal{P}_{f_o} (\alpha) = \frac{\epsilon_f \sqrt{N}}{2\pi i}\int_{(c)}\frac{L_\infty(1-s,f)}{L(s,f)}(N\alpha)^{-s}ds.
\end{equation}
We next shift the line of integration from $\re(s) = c$ to $\re(s) = \lambda$ where $1<\lambda<\frac{3}{2}$. Consider the positively oriented contour formed by the line segments $[\lambda-iT,\lambda+iT],\ [\lambda+iT,c+iT],\ [c+iT,c-iT]$ and $[c-it,\lambda-iT]$, where T is any positive real number. Clearly, the poles of the Gamma factors in $L_\infty(1-s, f)$ are on the right side of the line $\re(s) = \lambda$. Thus, the only poles inside the contour are arising from the non-trivial zeros $\rho$ of $L(s, f)$. Therefore, by invoking the Cauchy residue theorem, we arrive at
\begin{equation}\label{Cauchy residue theorem}
\int_{c-iT}^{c+iT}\frac{L_\infty(1-s,f)}{L(s,f)}(N\alpha)^{-s}ds = \left[\int_{c-iT}^{\lambda - iT} + \int_{\lambda - iT}^{\lambda + iT} + \int_{\lambda + iT}^{c + iT}\right]\frac{L_\infty(1-s,f)}{L(s,f)}(N\alpha)^{-s}ds - 2\pi i \sum_{|\im(\rho)|<T}R_{\rho},
\end{equation}
where $R_\rho$ denotes a residual term at a non-trivial zero $\rho$ of $L(s, f)$. Assuming that the non-trivial zeros of $L(s, f)$ are simple, the term $R_\rho$ evaluates as
\begin{align}\label{Evaluation of R}
R_\rho=\lim_{s \to \rho} (s-\rho) \frac{L_\infty(1-s,f)}{L(s,f)}(N\alpha)^{-s} =  \frac{L_\infty(1-\rho, f)}{L'(\rho, f)}(N\alpha)^{-\rho} = \frac{L_\infty(1-\rho, f)}{L'(\rho, f)}\beta^{\rho},
\end{align}
where in the last step we applied the relation $\alpha\beta = 1/N$. In general, if we assume the multiplicity of a non-trivial zero $\rho$ of $L(s, f)$ as $n_\rho$, then the residual term $R_\rho$ can be calculated as
\begin{align*}
R_{\rho}&=\frac{1}{(n_\rho-1)!}\lim_{s\to \rho}\frac{d^{n_\rho-1}}{ds^{n_\rho-1}}\left\{\frac{(s- \rho)^{n_\rho}L_\infty(1-s,f)(N\alpha)^{-s}}{L(s,f)}\right\}\nonumber\\
&=\frac{1}{(n_\rho-1)!}\lim_{s\to\rho}\frac{d^{n_\rho-1}}{ds^{n_\rho-1}}\left\{\frac{(s-\rho)^{n_\rho}L_\infty(1-s,f)\beta^{s}}{L(s,f)}\right\}.
\end{align*}

We can bound $\Gamma(s)$ for $s=\sigma+it$, in any vertical strip using Stirling's formula, which is given by (cf. \cite[p. 224]{Copson})
\begin{equation}\label{Stirling bound}
|\Gamma(s)| = (2\pi)^{1/2} |t|^{\sigma-\frac{1}{2}}e^{-\frac{1}{2} \pi |t|}\(1+ \mathcal{O}\(\frac{1}{|t|}\right)\right).
\end{equation}
Now, Lemma \ref{Lower bound of L(s, f)} and \eqref{Stirling bound} together implies that as $T\to \infty$,
\begin{equation*}
\frac{L_\infty(1-s,f)}{L(s,f)}(N\alpha)^{-s} = \mathcal{O}\left(e^{(A_2 - \frac{\pi}{2})|T|}\right),
\end{equation*}
where $A_2 < \pi/4$. Therefore the horizontal integrals in \eqref{Cauchy residue theorem} vanishes as $T\to \infty$. We next concentrate on the vertical integral in \eqref{Cauchy residue theorem} and denote the vertical integral as $\mathcal{V}$. Substituting $s$ by $1-s$ and applying the relation $\alpha \beta = 1/N$ respectively, the vertical integral $\mathcal{V}$ reduces to
\begin{align*}
\mathcal{V} := \int_{\lambda - iT}^{\lambda + iT} \frac{L_\infty(1-s,f)}{L(s,f)}(N\alpha)^{-s}ds = \beta \int_{1- \lambda - iT}^{1- \lambda + iT} \frac{L_\infty(s,f)}{L(1-s,f)}\beta^{-s}ds.
\end{align*}
Clearly $-\frac{1}{2}<1-\lambda<0$. Therefore as $T \to \infty$, we can apply \eqref{Integral representation of odd function} to evaluate the vertical integral $\mathcal{V}$  as
\begin{equation}\label{Vertical integral}
\mathcal{V} = 2\pi i \cdot 4 \pi \beta^2 \mathcal{P}_{f_o} (\beta).
\end{equation}
Combining \eqref{Apply functional equation}, \eqref{Cauchy residue theorem}, \eqref{Evaluation of R} and \eqref{Vertical integral} together, we have
\begin{equation*}
\frac{4\pi\alpha}{\epsilon_f \sqrt{N}} \mathcal{P}_{f_o} (\alpha) = 4 \pi \beta^2 \mathcal{P}_{f_o} (\beta) -  \sum_{\rho}\frac{L_\infty(1-\rho, f)}{L'(\rho, f)}\beta^{\rho}.
\end{equation*}
Finally, we multiply both sides of the above equation by $\frac{\epsilon_f\sqrt{N\alpha}}{4\pi}$ and invoke the relation $\alpha\beta=1/N$ to conclude the first part of our theorem.

For the second part, one can first apply Cauchy residue theorem in \eqref{imt of kinu} by shifting the line of integration from $c$ to $\lambda$ such that $-1 < \lambda < - \frac{1}{2}$ and substitute $x= \frac{2\pi \alpha}{n}$, $\mu = i \nu$ to obtain
\begin{align}\label{Inverse Mellin in even case}
2K_{i\nu}\left(\frac{2\pi\alpha}{n}\right)-\left(\frac{\pi \alpha}{n}\right)^{-i\nu}\Gamma(i\nu)-\left(\frac{\pi \alpha}{n}\right)^{i\nu}\Gamma(-i\nu)=\frac{1}{4\pi i}\int_{(\lambda)}\Gamma\left(\frac{s+i\nu}{2}\right)\Gamma\left(\frac{s-i\nu}{2}\right)\left(\frac{\pi \alpha}{n}\right)^{-s} ds.
\end{align}
Next, one can argue along the similar direction as was shown for part (a) to complete the proof. 

\section{Equivalent criterion for grand Riemann hypothesis when $f$ is odd}\label{Odd criterion}
We first provide a heuristic stemming from Theorem \ref{Hardy type identity}, which motivates us to get an equivalent criterion of the grand Riemann hypothesis for $L(s, f)$. The following asymptotic formula for the modified Bessel function $K_z(x)$ as $x\to0$ \cite[p.~375, Equations (9.6.8), (9.6.9)]{AS}:
\begin{align}\label{Bessel K in small argument}
K_z(x)\sim
\begin{cases}
\frac{1}{2}\Gamma(z)\left(\frac{x}{2}\right)^{-z}, & \mathrm{if}\ \mathrm{Re}(z)>0,\\
-\log(x),&\mathrm{if}\ z=0,
\end{cases}
\end{align}
and \eqref{Odd identity} together imply that 
\begin{align}\label{limit of p is zero}
\alpha^{3/2}P_{f_0}(\alpha)\to 0,
\end{align}
as $\alpha\to0$. We now assume the grand Riemann hypothesis for $L(s, f)$ and the convergence of the series $\sum_{\rho}\frac{L_\infty(1-\rho,f)}{L'(\rho,f)}\beta^{\rho-\frac{1}{2}}$. Then \eqref{Odd identity} and \eqref{limit of p is zero} together conclude
\begin{align*}
P_{f_0}(\beta)=\mathcal{O}\left(\beta^{-3/2}\right),
\end{align*}
as $\beta\to\infty$. 

The heuristic assumes the convergence of the series $\sum_{\rho}\frac{L_\infty(1-\rho,f)}{L'(\rho,f)}\beta^{\rho-\frac{1}{2}}$. It has been shown in Theorem \ref{Equivalent criterion} that the assumption of the grand Riemann hypothesis is enough to obtain the bound $P_{f_0}(\beta)=\mathcal{O}\left(\beta^{-3/2+\delta}\right)$, for any $\delta>0$. 

In the following lemma, we provide a bound for the function $P_{f_o}(y)$ as $y \to 0$.
\begin{lemma}\label{bound on P}
Let $-\frac{1}{2}<c<0$. We have
\begin{align*}
P_{f_o}(y)=\mathcal{O}\left(y^{-1-c}\right),
\end{align*}
as $y\to0$.
\end{lemma}
\begin{proof}
We first apply both \eqref{k mellin} and \eqref{Stirling bound} to obtain that for $-\frac{1}{2}<c<0$, 
\begin{align}\label{bound on K}
\frac{1}{n}K_{i\nu}\left(\frac{2\pi y}{n}\right)=\frac{1}{4\pi y}\frac{1}{2\pi i}\int_{(c)}L_\infty(s,f)\left(\frac{y}{n}\right)^{-s}ds 
\ll \frac{y^{-c-1}}{n^{-c}}.
\end{align}
Therefore, \eqref{Odd function} and \eqref{bound on K} imply that
\begin{align*}
P_{f_o}(y)\ll y^{-c-1}\sum_{n=1}^\infty\frac{\widetilde{\lambda}_f(n)}{n^{1-c}}
\ll y^{-c-1},
\end{align*}
where in the last step we used the convergence of the series $\sum_{n=1}^\infty\frac{\widetilde{\lambda}_f(n)}{n^{1-c}}$ for $-\frac{1}{2}<c<0$.
\end{proof}
We next obtain the Mellin transform of the function $P_{f_o}(y)$, which plays an important role to prove Theorem \ref{Equivalent criterion}, when $f$ is odd.
\begin{lemma}\label{MT of p}
For any complex number $s$ with $0<\mathrm{Re}(s)<\frac{1}{2}$, we have
\begin{align*}
\int_0^\infty y^{-s}P_{f_o}(y)dy=\frac{\pi^{s-1}}{4}\frac{\Gamma\left(\frac{1-s+i\nu}{2}\right)\Gamma\left(\frac{1-s-i\nu}{2}\right)}{L(f,s+1)}.
\end{align*}
\end{lemma}
\begin{proof}
Let
\begin{align*}
\phi(s):=\int_0^\infty  y^{-s}P_{f_o}(y)dy.
\end{align*}
We first make the change of variable $y$ to $x/n$ and then multiply both sides by $\lambda_f(n)/n$ to write the above integral as
\begin{align*}
\frac{\lambda_f(n)}{n^{s+1}}\phi(s)=\int_0^\infty  x^{-s} \frac{\lambda_f(n)}{n^2} P_{f_o}\(\frac{x}{n}\right) \, dx.
\end{align*}
Summing over $n$ from $1$ to $\infty$, we interchange the order of summation and integration by applying the Weierstrass M-test and the Lebesgue dominated convergence theorem to obtain
\begin{align}\label{lfpf}
L(s+1, f)\phi(s)=\int_0^\infty  x^{-s}\sum_{n=1}^\infty\frac{\lambda_f(n)}{n^2} P_{f_o}\(\frac{x}{n}\right) \, dx.
\end{align}
It follows from the integral representation of $P_{f_o}(y)$ as in \eqref{Integral representation of odd function} that for any $c \in (-\frac{1}{2}, 0)$, we have
\begin{align*}
\sum_{n=1}^\infty\frac{\lambda_f(n)}{n^2} P_{f_o}\(\frac{x}{n}\right) &=\sum_{n=1}^\infty\frac{\lambda_f(n)}{n^2} \frac{n}{4\pi x}\frac{1}{2\pi i}\int_{(c)}\frac{L_\infty(s,f)}{L(1-s,f)}\left(\frac{x}{n}\right)^{-s}ds\nonumber\\
&=\frac{1}{4\pi i x}\int_{(c)}L_\infty(s,f)x^{-s}ds,
\end{align*}
where in the last step we interchanged the order of summation and integration and used the series definition of $L(s, f)$. Therefore substituting $\alpha$ by $nx$ in the expression \eqref{k mellin}, the above series reduces to
\begin{align}\label{sum in k}
\sum_{n=1}^\infty\frac{\lambda_f(n)}{n^2}  P_{f_o}\(\frac{x}{n}\right) = K_{i\nu}(2\pi x).
\end{align}
Finally, we insert \eqref{sum in k} into \eqref{lfpf} and apply the Mellin transform of $K$-Bessel function for $0<\re(s)<\frac{1}{2}$ to conclude that
\begin{align*}
L(f,s+1)\phi(s)&=\int_0^\infty  x^{-s}K_{i\nu}(2\pi x)
=\frac{\pi^{s-1}}{4}\Gamma\left(\frac{1-s+i\nu}{2}\right)\Gamma\left(\frac{1-s-i\nu}{2}\right).
\end{align*}
This completes the proof of our lemma.
\end{proof}

We are now ready to prove an equivalent criterion of the grand Riemann hypothesis for $L(s, f)$ when $f$ is odd.

\subsection{Proof of Theorem \ref{Equivalent criterion}}

We first prove the sufficient condition of the grand Riemann hypothesis for $L(s,f)$ and for that we assume that the bound $P_{f_o}(y)=\mathcal{O}\left(y^{-\frac{3}{2}+\delta}\right)$ holds for any $\delta>0$, as $y\to \infty$. It is already known from Lemma \ref{MT of p} that the identity
\begin{align}\label{Odd identity in eqivalent criterion}
L(f,s+1) \int_0^\infty y^{-s}P_{f_o}(y)dy=\frac{\pi^{s-1}}{4}\Gamma\left(\frac{1-s+i\nu}{2}\right)\Gamma\left(\frac{1-s-i\nu}{2}\right).
\end{align}
holds in $0<\re(s)<\frac{1}{2}$. We claim that the identity is also valid inside the region $-\frac{1}{2}<\mathrm{Re}(s)\leq0$ under the assumption of the above bound of  $P_{f_o}(y)$.

To prove the claim, we first split the integral in the above identity into two parts, one from $0$ to $1$ and another from $1$ to $\infty$. Lemma \ref{bound on P} yields the convergence of the first integral inside the region $-\frac{1}{2}<\mathrm{Re}(s)\leq 0$. It follows from the bound $P_{f_o}(y)=\mathcal{O}\left(y^{-\frac{3}{2}+\delta}\right)$ that the second integral also converges inside the same region. Therefore, the result \cite[Theorem 3.2, p. 30]{temme} implies that the integral $ \int_0^\infty y^{-s}P_{f_o}(y) \, dy$ represents an analytic function in $-\frac{1}{2}<\mathrm{Re}(s)\leq0$. The gamma factors on the right hand side of \eqref{Odd identity in eqivalent criterion} are also analytic inside $-\frac{1}{2}<\mathrm{Re}(s)\leq0$ and $L(s+1, f)$ is entire. Therefore, the principle of analytic continuation yields our claim.  

Now the right-hand side of \eqref{Odd identity in eqivalent criterion} has no zeros in $-\frac{1}{2}<\mathrm{Re}(s)\leq0$. On the other hand, the integral on the left-hand side is analytic inside the same region. Therefore, $L(s+1, f)$ does not vanish in $-\frac{1}{2}<\mathrm{Re}(s)\leq0$ which concludes that the grand Riemann hypothesis is true.

We next prove the converse part and for that we first assume that the grand Riemann hypothesis is true. Under the assumption of the grand Riemann hypothesis for $L(s, f)$, it follows from \cite[Proposition 5.14, p. 113]{Iwaniec} that as $x \to \infty$,
\begin{align*}
\sum_{n\leq x}\widetilde{\lambda}_f(n)=\mathcal{O}\left(x^{\frac{1}{2}+\delta}\right),
\end{align*}
where $\delta$ is any positive real number. Let
\begin{align*}
M(\ell,n):=\sum_{m=\ell}^n\frac{\widetilde{\lambda}_f(m)}{m}.
\end{align*}
Applying the Euler's partial summation formula, the above function can be bounded  as
\begin{equation}\label{Bound of M(l, n)}
M(\ell,n) = \mathcal{O}\left(\ell^{-1/2+\delta}\right).
\end{equation}
Let $\ell=\lfloor y^{1-\delta}\rfloor$. We split the sum $P_{f_o}(y)$ into two parts as
\begin{align}\label{p in s1 s2}
P_{f_o}(y)=S_1(f, y)+S_2(f, y),
\end{align}
where
\begin{align*}
S_1(f,y):=\sum_{n=1}^{\ell-1}\frac{\widetilde{\lambda}_f(n)}{n^2}K_{i\nu}\left(\frac{2\pi y}{n}\right),
\end{align*}
and 
\begin{align*}
S_2(f,y):=\sum_{n=\ell}^\infty\frac{\widetilde{\lambda}_f(n)}{n^2}K_{i\nu}\left(\frac{2\pi y}{n}\right).
\end{align*}
We first estimate $S_2(f,y)$. We have
\begin{align*}
\sum_{n=\ell}^N\frac{\widetilde{\lambda}_f(n)}{n^2}K_{i\nu}\left(\frac{2\pi y}{n}\right)=\sum_{n=\ell}^{N-1}M(\ell,n)\left(\frac{K_{i\nu}\left(\frac{2\pi y}{n}\right)}{n}-\frac{K_{i\nu}\left(\frac{2\pi y}{n+1}\right)}{n+1}\right)+M(\ell,N)\frac{K_{i\nu}\left(\frac{2\pi y}{N}\right)}{N},
\end{align*}
where the last term vanishes as $N\to\infty$. Therefore, the above equation implies
\begin{align}\label{S2fy}
S_2(f,y)=\sum_{n=\ell}^{\infty}M(\ell,n)\left(\frac{K_{i\nu}\left(\frac{2\pi y}{n}\right)}{n}-\frac{K_{i\nu}\left(\frac{2\pi y}{n+1}\right)}{n+1}\right).
\end{align}
We next estimate the above sum by utilizing the mean value theorem and for that we define
\begin{align}\label{defn g}
g(x):&=\frac{K_{i\nu}\left(\frac{2\pi y}{x}\right)}{x}.
\end{align}
Invoking \eqref{k mellin} into \eqref{defn g}, we obtain
\begin{align*}
g(x)=\frac{1}{4\pi y}\frac{1}{2\pi i}\int_{(c)}\Gamma\left(\frac{s+1+i\nu}{2}\right)\Gamma\left(\frac{s+1-i\nu}{2}\right)\left(\frac{\pi y}{x}\right)^{-s}ds,
\end{align*}
where $-\frac{1}{2}< c <0$. We differentiate $g(x)$ with respect to $x$ and apply Stirling's formula \eqref{Stirling bound} on the gamma factors to derive that
\begin{align}\label{g bound}
g'(x)&=\frac{1}{4\pi y}\int_{(c)}\Gamma\left(\frac{s+1+i\nu}{2}\right)\Gamma\left(\frac{s+1-i\nu}{2}\right)s\left(\pi y\right)^{-s}x^{s-1}ds 
\ll \frac{x^{c-1}}{y^{c+1}}.
\end{align}
It follows from the mean value theorem that there exist $\lambda_n\in(n,n+1)$ such that $g(n)-g(n+1)=-g'(\lambda_n)$. Therefore the sum $S_2(f, y)$ in \eqref{S2fy} can be written as
\begin{align*}
S_2(f,y)&=\sum_{n=\ell}^{\infty}M(\ell,n)(g(n)-g(n+1)) 
= \sum_{n=\ell}^{\infty}M(\ell,n)(-g'(\lambda_n)).
\end{align*}
Inserting the bounds from \eqref{Bound of M(l, n)} and \eqref{g bound} into the above equation, we arrive at
\begin{align}\label{s2 final bound odd}
S_2(f,y)&\ll \frac{\ell^{-\frac{1}{2}+\delta}}{y^{c+1}}\sum_{n=\ell}^\infty n^{c-1}\ll \frac{\ell^{c-\frac{1}{2}+\delta}}{y^{c+1}}
\ll y^{-\frac{3}{2}+\delta},
\end{align}
where in the penultimate step we have used the fact that $\sum_{n>x} n^{-s} = \mathcal{O}(x^{1-s})$ (cf. \cite[p.~55, Theorem 3.2(c)]{apostal}) and in the last step, we put $\ell=\lfloor y^{1-\delta}\rfloor$.

Next, we will concentrate on $S_1(f,y)$. Utilizing the asymptotics of $K_{i\nu}(z)$ as $z \to \infty$ \cite[Formula 10.40.2, p. 255]{NIST}, we can bound $S_1(f,y)$ as
\begin{align*}
S_1(f,y)\ll \frac{e^{-\frac{2\pi y}{\ell}}}{\sqrt{y}}\sum_{n=1}^{\ell-1}\frac{\widetilde{\lambda}_f(n)}{n^{3/2}}.
\end{align*}
Therefore, the bound of $\widetilde{\lambda}_f(n)$, followed from \eqref{Sarnak bound} and the value $\ell=\lfloor y^{1-\delta}\rfloor$, yields
\begin{equation}\label{s1 final bound}
S_1(f,y) \ll \frac{e^{-\frac{2\pi y}{\ell}}}{\sqrt{y}}\ell^{-\frac{25}{64}+\delta}
\ll y^{-\frac{57}{64}+\delta}e^{-2\pi y^\delta},
\end{equation}
for any $\delta > 0$. Combining the estimates of $S_1(f,y)$ and $S_2(f,y)$ as in \eqref{s1 final bound} and \eqref{s2 final bound odd} respectively into \eqref{p in s1 s2}, we can conclude that as $y\to \infty$, $P_{f_o}(y)=\mathcal{O}\left(y^{-\frac{3}{2}+\delta}\right)$ for any $\delta>0$. This completes the proof of Theorem \ref{Equivalent criterion}(a).

\section{Equivalent criterion for grand Riemann hypothesis when $f$ is even}\label{Even criterion}

In this section, we provide a criterion of the grand Riemann hypothesis for $L(s, f)$ when $f$ is even, which is motivated from Theorem \ref{Hardy type identity}(b). Under the assumption of the grand Riemann hypothesis for $L(s, f)$ and the convergence of the series $\sum_{\rho}\frac{L_\infty(1-\rho,f)}{L'(\rho,f)}\beta^{\rho-\frac{1}{2}}$, it follows from Theorem  \ref{Hardy type identity} and \eqref{Bessel K in small argument} that as $\alpha \to 0$, or equivalently, as $\beta \to \infty$,
\begin{equation*}
\mathcal{P}_{f_e}(\beta) = \frac{\pi^{i\nu}\Gamma(-i\nu)}{L(-i\nu,f)}\beta^{i \nu} +\frac{\pi^{-i\nu}\Gamma(i\nu)}{L(i\nu,f)}\beta^{-i\nu} + \mathcal{O}(\beta^{-1/2})
\end{equation*}
The heuristic is true under the assumption of the convergence of the series $\sum_{\rho}\frac{L_\infty(1-\rho,f)}{L'(\rho,f)}\beta^{\rho-\frac{1}{2}}$ but without this assumption, it is shown in Theorem \ref{Equivalent criterion} that the estimate of $\mathcal{P}_{f_e}(\beta)$ provides the main term of order $\beta^{|i\nu|}$ and the error term of the order $\beta^{-1/2+\delta}$ for any $\delta>0$. As mentioned earlier in section \S \ref{Introduction}, we need to consider $\mathcal{Q}_{f_e}(\beta)$ to find the criterion of the grand Riemann hypothesis for $L(s, f)$.

In order to prove the sufficient condition of the grand Riemann hypothesis, the following two lemmas are crucial.
\begin{lemma}\label{Q bound}
Let $-\frac{3}{2}<c<-\frac{1}{2}$. Then for any $y>0$, 
\begin{align*}
\mathcal{Q}_{f_e}(y)=\mathcal{O}\left(y^{-c-1}\right).
\end{align*}
\end{lemma}
\begin{proof}
We first consider the function
\begin{align}\label{deltayn}
\Omega_\nu(y, x):=2K_{1+i\nu}\left(\frac{2\pi y}{x}\right)+2K_{1-i\nu}\left(\frac{2\pi y}{x}\right) -\frac{\Gamma(1+i\nu)}{(\pi y/x)^{1+i \nu}} - \frac{\Gamma(1-i\nu)}{(\pi y/x)^{1-i \nu}}.
\end{align}
The above function can be written as
\begin{align}\label{Mellin transform of Delta}
\Omega_\nu(y, x)&=-\frac{x}{\pi}\frac{d}{dy}\left[2K_{i\nu}\left(\frac{2\pi y}{x}\right)-\frac{\Gamma(i\nu)}{(\pi y/x)^{i\nu}} - \frac{\Gamma(-i\nu)}{(\pi y/x)^{-i\nu}} \right]\nonumber\\
&=-\frac{x}{\pi}\frac{d}{dy}\frac{1}{4\pi i}\int_{(c)}\Gamma\left(\frac{s+i\nu}{2}\right)\Gamma\left(\frac{s-i\nu}{2}\right)\left(\frac{\pi y}{x}\right)^{-s}ds\nonumber\\
&=\frac{1}{4\pi i}\int_{(c)}s\Gamma\left(\frac{s+i\nu}{2}\right)\Gamma\left(\frac{s-i\nu}{2}\right)\left(\frac{\pi y}{x}\right)^{-s-1}ds,
\end{align}
where in the penultimate step we have applied \eqref{Inverse Mellin in even case}. Invoking Stirling formula for gamma functions in the above integral, we obtain
\begin{align}\label{kbound}
\Omega_\nu(y, x)=\mathcal{O}\left((x/y)^{c+1}\right).
\end{align}
Therefore, the definition of $\mathcal{Q}_{f_e}(y)$ together with \eqref{kbound} yields
\begin{align*}
\mathcal{Q}_{f_e}(y)&\ll y^{-1-c}\sum_{n=1}^\infty\frac{\widetilde{\lambda}_f(n)}{n^{1-c}}\ll y^{-c-1},
\end{align*}
where in the last step we used the fact that the series $\sum_{n=1}^\infty\frac{\widetilde{\lambda}_f(n)}{n^{1-c}}$ converges as $-\frac{3}{2}<c<-\frac{1}{2}$. This proves the lemma.
\end{proof}
In the following lemma, we provide the Mellin transform of $\mathcal{Q}_{f_e}(y)$, which plays an important role to prove the sufficient condition of the grand Riemann hypothesis for $L(s,f)$, when $f$ is even.
\begin{lemma}\label{Mellin transform}
For any complex number $s$ with $\frac{1}{2}<\mathrm{Re}(s)<\frac{3}{2}$, we have
\begin{align}\label{Mellin transform for even function}
\int_0^\infty y^{-s}\mathcal{Q}_{f_e}(y)dy=\frac{\pi^ss}{2}\frac{\Gamma\left(\frac{-s+i\nu}{2}\right)\Gamma\left(\frac{-s-i\nu}{2}\right)}{L(f,s+1)}.
\end{align}
\end{lemma}
\begin{proof}
Proceeding similarly as was done in the proof of Lemma \ref{MT of p}, we can obtain
\begin{align*}
\int_0^\infty y^{-s-1}\mathcal{P}_{f_e}(y)dy &=  \frac{\pi^{s} \, \Gamma\left(\frac{-s+i\nu}{2}\right)\Gamma\left(\frac{-s-i\nu}{2}\right)}{2 \, L(s+1, f)}.
\end{align*}
We finally perform the integration by parts and utilize the fact that $\mathcal{Q}_{f_e}(y) = \frac{d}{dy} \mathcal{P}_{f_e}(y)$ to conclude our lemma.
\end{proof}

\subsection{Proof of Theorem \ref{Equivalent criterion even}}

We first prove part (a) and for that we assume the bound $\mathcal{Q}_{f_e} (y) = \mathcal{O}\(y^{-3/2+\delta}\right)$ for any $\delta>0$. Multiplying both sides of \eqref{Mellin transform for even function} with $\left(\frac{s^2+\nu^2}{4}\right)$, the identity
\begin{align}\label{deciding identity}
L(f,s+1)\left(\frac{s^2+\nu^2}{4}\right)\int_0^\infty y^{-s}\mathcal{Q}_{f_e} (y) \, dy=\frac{\pi^ss}{2}\Gamma\left(1+\frac{-s+i\nu}{2}\right)\Gamma\left(1+\frac{-s-i\nu}{2}\right)
\end{align}
holds for $\frac{1}{2}<\mathrm{Re}(s)<\frac{3}{2}$. Invoking Lemma \ref{Q bound} and the bound of $\mathcal{Q}_{f_e} (y)$, it can be shown by the similar argument of the proof of Theorem \ref{Equivalent criterion} that the identity is also valid inside the region $-\frac{1}{2}<\mathrm{Re}(s)<\frac{3}{2}$.

Now, the right-hand side of \eqref{deciding identity} has no zeros in $-\frac{1}{2}<\mathrm{Re}(s)<0$. On the other hand, integral is also analytic inside the same region. Therefore, $L(f,s+1)$ has no zeros in $-\frac{1}{2}<\mathrm{Re}(s)<0$. This completes the proof of the grand Riemann hypothesis. 

We next prove part (b) and for that we assume the grand Riemann hypothesis for $L(s, f)$. The argument to prove the estimates of  
$\mathcal{P}_{f_e} (y)$ and $\mathcal{Q}_{f_e} (y)$ are similar, we here provide the proof for $\mathcal{Q}_{f_e} (y)$.

It follows from \cite[Proposition 5.14, p. 113]{Iwaniec} that the grand Riemann hypothesis for $L(s, f)$ implies
\begin{align*}
\sum_{n\leq x}\widetilde{\lambda}_f(n)=\mathcal{O}\left(x^{\frac{1}{2}+\delta}\right),
\end{align*}
as $x \to \infty$, where $\delta$ is any positive real number. Let
\begin{align*}
M(\ell,n):=\sum_{m=\ell}^n\frac{\widetilde{\lambda}_f(m)}{m^2}.
\end{align*}
Applying the Euler's partial summation formula, the above function can be bounded  as
\begin{equation}\label{Bound of M(l, n) even}
M(\ell,n) = \mathcal{O}\left(\ell^{-3/2+\delta}\right).
\end{equation}
Inserting $\Omega_\nu(y,x)$ from \eqref{deltayn} into the definition of $\mathcal{Q}_{f_e} (y)$ in \eqref{Q function}, we have
\begin{equation*}
\mathcal{Q}_{f_e} (y) = \pi \sum_{n=1}^\infty \frac{\widetilde{\lambda}_f(n)}{n^2} \Omega_\nu(y,n).
\end{equation*}
Let $\ell=\lfloor y^{1-\delta}\rfloor$. We next split the above sum into two parts as
\begin{align}\label{qfy in terms s1fy s2fy}
\mathcal{Q}_{f_e} (y) &= \sum_{n=1}^{\ell-1}\frac{\widetilde{\lambda}_f(n)}{n^2}\Omega_\nu(y,n)+\sum_{n=\ell}^{\infty}\frac{\widetilde{\lambda}_f(n)}{n^2}\Omega_\nu(y,n)\nonumber\\
&=: T_1(f,y) + T_2(f,y).
\end{align}
We first evaluate $T_2(f,y)$. We have
\begin{align*}
\sum_{n=\ell}^{N}\frac{\widetilde{\lambda}_f(n)}{n^2}\Omega_\nu(y,n)=\sum_{n=\ell}^{N-1}M(\ell,n) \left(\Omega_\nu(y,n)-\Omega_\nu(y,n+1)\right)-M(\ell,N)\Omega_\nu(y,N).
\end{align*}
Note that the last term in the above equation vanishes as $N\to\infty$, therefore, we obtain
\begin{align}\label{before mvt}
T_2(f,y)&=\sum_{n=\ell}^{\infty}M(\ell,n) \left(\Omega_\nu(y,n)-\Omega_\nu(y,n+1)\right).
\end{align}
The mean value theorem implies that there exists $\lambda_n\in(n,n+1)$  such that 
\begin{align}\label{mvt}
\Omega_\nu(y,n)-\Omega_\nu(y,n+1)=-\frac{d}{dx}\Omega_\nu(y,x)\bigg|_{x= \lambda_n}.
\end{align}
We next find the estimate for the expression on the right-hand side of the above equation. Differentiating both sides of \eqref{Mellin transform of Delta} with respect to $x$, we arrive at
\begin{align}\label{bound on f'}
\frac{d}{dx}\Omega_\nu(y,x) &=- \frac{1}{4 \pi^2 i y}\int_{(c)}s(s+1)\Gamma\left(\frac{s+i\nu}{2}\right)\Gamma\left(\frac{s-i\nu}{2}\right) \left(\frac{x}{\pi y}\right)^{s} \, ds
\ll y^{-c-1}x^c.
\end{align}
The combination of \eqref{Bound of M(l, n) even}, \eqref{before mvt}, \eqref{mvt} and \eqref{bound on f'} together with $\ell=\lfloor y^{1-\epsilon}\rfloor$, yields
\begin{align}\label{final s2fy}
T_2(f,y)&\ll \ell^{-\frac{3}{2}+\epsilon}y^{-c-1}\sum_{n=\ell}^\infty n^c
\ll y^{-\frac{3}{2}+\epsilon},
\end{align}
where in the last step we used the fact $\sum_{n>x}n^{-s}=\mathcal{O}(x^{1-s})$ (cf. \cite[p.~55, Theorem 3.2(c)]{apostal}).

We now estimate $T_1(f,y)$. Invoking the estimate of $K_{z}(x)$ from \cite[p.~255, Formula 10.40.2]{NIST}, we arrive at
\begin{align*}
T_1(f,y)= - \sum_{n=1}^{\ell-1}\frac{\widetilde{\lambda}_f(n)}{n^2} \left\lbrace \(\frac{\pi y}{n}\right)^{i\nu-1}\Gamma(1-i\nu)+\(\frac{\pi y}{n}\right)^{-i\nu-1}\Gamma(1+i\nu) \right\rbrace+\mathcal{O}\left(\frac{e^{-\frac{2\pi y}{\ell}}}{\sqrt{y}}\sum_{n=1}^{\ell-1}\frac{\widetilde{\lambda}_f(n)}{n^{3/2}}\right).
\end{align*}
Therefore, the bound of $\widetilde{\lambda}_f(n)$, followed from \eqref{Sarnak bound} and the value $\ell=\lfloor y^{1-\delta}\rfloor$, yields 
\begin{equation}\label{s1 final bound odd}
T_1(f,y)= - \sum_{n=1}^{\ell-1}\frac{\widetilde{\lambda}_f(n)}{n^2} \left\lbrace \(\frac{\pi y}{n}\right)^{i\nu-1}\Gamma(1-i\nu)+\(\frac{\pi y}{n}\right)^{-i\nu-1}\Gamma(1+i\nu) \right\rbrace+\mathcal{O}\left(y^{-\frac{57}{64}+\delta}e^{-2\pi y^\delta}\right),
\end{equation}
for any $\delta > 0$. Thus combining \eqref{qfy in terms s1fy s2fy}, \eqref{final s2fy} and \eqref{s1 final bound odd}, we conclude \eqref{Bound on Q}. This completes the proof of our theorem. 
\vspace{.3cm}

\subsection*{Acknowledgements} 
The authors would like to show their sincere gratitude to Prof. Atul Dixit for fruitful discussions and suggestions on the manuscript. The first author's research was supported by the SERB-DST CRG grant CRG/2020/002367 and the second author's research was supported by the grant IBS-R003-D1 of the IBS-CGP, POSTECH, South Korea.


\begin{thebibliography}{99}

\bibitem{AS} M.~Abramowitz and I.~A.~Stegun, {\em Handbook of Mathematical Functions, with Formulas, Graphs and Mathematical Tables}, 9th ed., Dover, New York, 1970.

\bibitem{Agarwal} A.~Agarwal, M.~Garg and B.~Maji, {\em Riesz-type criteria for the Riemann Hypothesis}, submitted for publication.

\bibitem{Liu}  T.~Aoki, S.~Kanemitsu and J.~Liu, {\em Number Theory: Dreaming In Dreams - Proceedings Of The 5th China-japan Seminar}, World Scientific (2009).

\bibitem{apostal} T.~M.~Apostol, \emph{Introduction to Analytic Number Theory}, Undergraduate Texts in Mathematics, New York-Heidelberg, Springer-Verlag, 1976.

\bibitem{Berndt} B. C. Berndt, {\em Ramanujan’s Notebooks, Part V}, Springer-Verlag, New York, 1998.

\bibitem{Brychkov} A. Y. Brychkov, {\em Handbook of special functions: derivatives, integrals, series and other formulas}, CRC press, 2008.

\bibitem{Copson} E. T. Copson, {\em Theory of Functions of a Complex Variable}, Oxford University Press, Oxford, 1935.

\bibitem{Dix} A.~Dixit, {\em Character analogues of Ramanujan-type integrals involving the Riemann $\Xi$-function}, Pacific J. Math. {\bf 255}, No. 2 (2012), 317--348.

\bibitem{Dixit3} A.~Dixit, S.~Gupta and A.~Vatwani, {\em A modular relation involving non-trivial zeros of the Dedekind zeta function, and the generalized Riemann hypothesis}, submitted for publication.

\bibitem{Dixit} A.~Dixit, A.~Roy and A.~Zaharescu, {\em Ramanujan-Hardy-Littlewood-Riesz phenomena for Hecke forms}, J. Math. Anal. Appl. {\bf 426} (2015), 594--611.

\bibitem{Dixit2} A.~Dixit, A.~Roy and A.~Zaharescu, {\em Riesz-type criteria and theta transformation analogues}, J. Number Theory {\bf 160} (2016), 385--408.

\bibitem{Hardy} G.~H.~Hardy and J.~E.~Littlewood, {\em Contributions to the theory of the Riemann zeta-Function and the theory of the distribution of primes}, Acta Math., {\bf 41} (1916), 119--196.

\bibitem{Iwaniec} H. Iwaniec and E. Kowalski, {\em Analytic Number Theory}, American Mathematical Society Colloquium Publications, {\bf 53}, American Mathematical Society, Providence, RI, 2004.

\bibitem{Kim} H.~Kim and P.~Sarnak, {\em Refined estimates towards the Ramanujan and Selberg conjectures}, J. Amer. Math. Soc. {\bf 16} (2003), 175--181.

\bibitem{NIST} F.~W.~J.~Olver, D.~W.~Lozier, R.~F.~Boisvert and C.~W.~Clark, eds., {\em NIST Handbook of Mathematical Functions}, Cambridge University Press, Cambridge, 2010.

\bibitem{Ramanujan} S.~Ramanujan, {\em Notebooks of Ramanujan, Vol 2}, Tata Institute of FundamentaI Research, Bombay, 1957.

\bibitem{Riesz} M.~Riesz, {\em Sur l'hypoth{\`e}se de Riemann}, Acta Math., {\bf 40} (1916), 185--190.

\bibitem{Sankaranarayanan} A.~Sankaranarayanan and J.~Sengupta, {\em Zero-density estimate of $L$-functions attached to Maass forms}, Acta Arith., {\bf 127} (2007), 273--284.


\bibitem{temme} N.~M.~Temme, \emph{Special functions: An introduction to the classical functions of mathematical physics}, Wiley-Interscience Publication, New York, 1996.

\bibitem{Titchmarsh} E.~C.~Titchmarsh, {\em The theory of the Riemann zeta function}, Clarendon Press, Oxford, 1986.

\end{thebibliography}
\end{document}